\documentclass[10pt]{amsart}

\usepackage{amssymb}      
\usepackage{amsmath}      
\usepackage{amsthm}      
\usepackage{amsbsy}      
\usepackage{amscd}      
\usepackage[dvips]{graphicx}

\theoremstyle{plain}      
\newtheorem{theorem}{Theorem}[section]      
\newtheorem{lemma}{Lemma}[section]      
\newtheorem{corollary}[theorem]{Corollary}      
\newtheorem{proposition}{Proposition}[section]

\newtheorem{definition}{Definition}[section]          
\theoremstyle{remark}      
\newtheorem{remark}{Remark}[section]  
   
\newtheorem*{acknowledgements}{Acknowledgements}

\newcommand{\Q}{{\mathbb{Q}}}        
\newcommand{\Z}{{\mathbb{Z}}}   
   
\newcommand{\C}{{\mathbb{C}}}      
\newcommand{\R}{{\mathbb{R}}}

\begin{document}

\date{\today}

\title[Two questions on mapping class groups]{Two questions on mapping class groups}
 \author[L.Funar]{Louis Funar}     
\address{Institut Fourier BP 74, UMR 5582, 
University of Grenoble I, 
38402 Saint-Martin-d'H\`eres cedex, France}      
\email{funar@fourier.ujf-grenoble.fr}

\begin{abstract}
We show that central extensions of the 
mapping class group $M_g$ of the closed orientable surface of genus $g$ 
by $\Z$ are residually finite. 
Further we give rough estimates 
of the largest $N=N_g$ such that homomorphisms 
from $M_g$  to $SU(N)$ have finite image. 
In particular, homomorphisms of $M_g$ into $SL([\sqrt{g+1}],\C)$ 
have finite image. 
Both results come from properties of quantum 
representations of mapping class groups.

\vspace{0.1cm}
 

\end{abstract}

\maketitle

\section{Introduction and statements}

Set $\Sigma_{g}^r$ for the 
orientable surface of genus $g$ with $r$ 
punctures. We denote by 
$M_{g}^r$ the mapping class group of $\Sigma_{g}^r$, namely 
the group of isotopy classes of homeomorphisms that fix 
the punctures.

The following  answers Question 6.4 of Farb (see Chapter 2 of \cite{Fa}).

\begin{proposition}\label{res}
The central extensions of the mapping class group $M_g$ (or the punctured 
mapping class group $M_{g}^1$, for $g\geq 4$) by $\Z$ are residually finite.
\end{proposition}

\begin{remark}
The universal central extension $\widetilde{M}_g(1)$ surjects onto the 
universal central extension  $\widetilde{Sp(2g,\Z)}$ of 
the (integral) symplectic group, whose class is the Maslov class 
(generating $H^2(Sp(2g,\Z))$.  
It is known that  $\widetilde{Sp(2g,\Z)}$ is the pull-back of 
$Sp(2g,\Z)$ into the universal covering  $\widetilde{Sp(2g,\R)}$
of the real symplectic group.

By a result of Deligne (see \cite{De}) the extension  
$\widetilde{Sp(2g,\Z)}$, for $g\geq 2$, is {\em not} residually 
finite since any finite index subgroup of it contains 
$2\Z$, where $\Z$ is the central kernel  
$\widetilde{Sp(2g,\Z)}\to Sp(2g,\Z)$. 
The same holds, more generally, for some other 
arithmetic groups having the congruence subgroup property.    
\end{remark}

The method of proof uses quantum representations of mapping class groups.

\begin{definition}
The group $\Gamma$ has property $F_n$ if all homomorphisms 
$\Gamma\to PU(n)$ have finite image. Moreover, the group $\Gamma$ 
has property $F$ if it has property $F$ for every $n$. 
\end{definition}

Observe that the property $F$ is inherited by 
finite index subgroups. 

\begin{remark}
Let $G$ be a connected, semi-simple, 
almost $\Q$-simple algebraic $\Q$-group and $\Gamma$ an arithmetic 
lattice in $G$. If $G_{\R}$ has real rank at least 2 and $G^{ad}(\R)$ 
has no compact factor then $\Gamma$ has property $F$. This follows from 
(\cite{Mar}, Chap. VIII, Thm.B) for $K=\Q, l=\R$, $S$ containing only 
the Archimedian place of $\Q$ and ${\bf H}=PO(n)$. 
In particular, any discrete group $\Gamma$ commensurable with $Sp(2g,\Z)$ 
for $g\geq 2$ or to $SL(2,{\mathcal O})$, 
where ${\mathcal O}$ is the ring of integers 
in a totally real  number field of degree at least 2, has property $F$. 
\end{remark}

Mapping class groups have not property $F$. It is therefore interesting 
to understand whether they have property $F_n$ for some $n$. 
This is related to a question of Farb in \cite{Fa} 
concerning linear representations in small degree.  
The previous remark shows that we cannot use 
unitary representations of $M_g$ that factor 
through $Sp(2g,\Z)$,  as the later group has no finite dimensional 
unitary representations with infinite image.   
Our second result states as follows:   

\begin{proposition}\label{fa}
The maximal number $N_g$ for which  $M_g$ has property $F_{N_g}$ satisfies: 
\[ \sqrt{g+1} \leq N_g < \left\{\begin{array}{ll}
5^{g/2}F_{g-1}, \mbox{\rm if } g \:\mbox{\rm is even},\\
5^{(g-1)/2}(F_{g}+F_{g-2}), \mbox{\rm if } g \:\mbox{\rm is odd},
\end{array}\right.\]
where $F_j$ is the Fibonacci sequence, defined by 
$F_0=0,F_1=1$ and the recurrence $F_{n+1}=F_n+F_{n-1}$, for $n\geq 1$. 
Moreover, the upper bounds are valid for any finite index subgroup 
of $M_g$. 
\end{proposition}

\begin{corollary}\label{fin}
Every homomorphism $M_g\to SL([\sqrt{g+1}],\C)$ has finite image, 
if $g\geq 1$.
\end{corollary}

It is likely that $N_g$ behaves like an exponential for large $g$. 
This seems difficult to check because 
very few unitary representations of $M_g$ are known. 
On the other hand one might expect that the maximal 
$n$ with the property than every 
homomorphism $M_g\to SL(n,\C)$ has finite image 
is  a linear  function on $g$.

Notice that groups having  homomorphisms with 
infinite image into $SL(2,\C)$ have not the 
property $T$ of Kazhdan. However, $M_g$ has no such representations, 
if $g\geq 3$, by the Corollary above.

Results of similar flavor were proved in \cite{FLM} where it is shown 
that representations $M_g\to GL(2\sqrt{g-1},\C)$ cannot be faithful 
and in \cite{B2} where it is shown that  the image of an element 
of $M_g$ under a representation into $GL(g,\C)$ should have algebraic 
eigenvalues.

One inequality above is an immediate consequence of a theorem of Bridson 
(\cite{B2}) concerning the property $FA_n$, which was 
introduced by Farb in \cite{Fa2}. 
The second inequality comes from the existence of quantum representations 
of $M_g$ with infinite image (\cite{F}).

\section{Proof of Proposition \ref{res}}

We prove the claim for the universal central extension first. 
This is known when $g=1$ since the universal 
central extension is isomorphic to the braid group $B_3$.

An important result  due independently to Andersen (\cite{A}) and 
to Freedman, Walker and Zhang  (\cite{FWW}) states that the
$SU(2)$ TQFT representation of the mapping class group
is {\em asymptotically faithful}. Specifically, there is a sequence
of  representations $\rho_k$ (indexed by an integer $k$, called the level)
$\rho_k:M_g\to PU(N(k,g))$ into the projective unitary group 
of dimension $N(k,g)$ (for some $N(k,g)$ depending exponentially on $k$)
such that the intersection of  the kernels $\cap_{k\geq 2}\ker \rho_k$ is 
trivial for $g\geq 3$, and respectively the center of the mapping class group 
$M_2$ (which is a group of order two generated 
by hyperelliptic involution), when $g=2$. Moreover, for $g=2$ 
we can use the $SU(n)$ TQFT representation, with $n\geq 3$,  
for  which the intersection of the kernels above is trivial (see \cite{A}). 
When using this result we will say that we make use of 
the asymptotic faithfulness (of 
the quantum representations).

Each  quantum representation is a projective 
representation which lifts to a linear
representation $\widetilde{\rho}_k: \widetilde{M}_g(12)\to U(N(k,g))$ of the 
central extension $\widetilde{M}_g(12)$ of the mapping class 
group $M_g$ by $\Z$. 
The later representation corresponds to invariants of 3-manifolds 
with a $p_1$-structure. Masbaum, Roberts (\cite{MR}) and 
Gervais (\cite{Ge}) gave a precise  description 
of this extension.  Namely, the 
cohomology class $c_{\widetilde{M}_g(12)}\in 
H^2(M_g,\Z)$  associated to this extension equals 
12 times the  signature class $\chi$. 
It is known (see \cite{KS}) that the group $H^2(M_g)$ is generated by 
$\chi$, when $g\geq 2$.
Recall that $\chi$ is the class of one fourth the Meyer signature cocycle.

Observe that the $\rho_k$ action of the center of $\widetilde{M}_g(12)$ is by
roots of unity of order $2k$ (see \cite{MR} for the explicit formula). 
In fact, this action corresponds to the change of the $p_1$-structure of a 
3-manifold and it is well-known that the quantum invariant changes by a 
root of unity of order $2k$.  
Thus  every element of the center acts non-trivially via $\widetilde{\rho}_k$,  
for large enough $k$, so that 
the representations of $\widetilde{M}_g(12)$ are also asymptotically faithful.
This implies that $\widetilde{M}_g(12)$ is residually finite.  
In fact, let $a\in \widetilde{M}_g(12)$ be any element $a\neq 1$. 
By the asymptotic faithfulness 
there exists some level $k$ so that 
$\widetilde{\rho}_k(a)\in U(N(k,g))$ is non-trivial. 
The subgroup $\widetilde{\rho}_k(\widetilde{M}_g(12))\subset U(N(k,g))$ is 
a discrete linear group and thus, by a classical theorem of Malcev, it 
is residually finite. In particular, there exists a homomorphism of 
$\widetilde{\rho}_k(\widetilde{M}_g(12))$ onto some finite group  
sending $\widetilde{\rho}_k(a)$ into a non-trivial element. 
This shows that every 
non-trivial element of $\widetilde{M}_g(12)$ is detected by some  
homomorphism into some finite group.

The universal central extension  is $\widetilde{M}_g(1)$, 
where $\widetilde{M}_g(n)$ denotes the central extension 
by $\Z$ whose class is $c_{\widetilde{M}_g(n)}=n\chi$. 
It is immediate from  their explicit presentations (see \cite{Ge})
that $\widetilde{M}_g(d)$ embeds 
into $\widetilde{M}_g(n)$ if $d$ divides $n\neq 0$. Such an embedding sends the 
generator $z$ of the center into $z^{n/d}$.  
In particular,  $\widetilde{M}_g(1)$ embeds in $\widetilde{M}_g(12)$ and thus 
the universal central extension is residually finite.

Now, an arbitrary central extension of $M_g$ by $\Z$ 
is either trivial and hence 
residually finite, or else isomorphic to 
$\widetilde{M}_g(n)$, for some $n\in \Z\setminus\{0\}$.  
We observed above that there is an injective 
homomorphism $\widetilde{M}_g(1)\to \widetilde{M}_g(n)$, which sends the 
central element $z$ into $z^n$. Moreover, the image is a normal 
subgroup of $\widetilde{M}_g(n)$. In particular, we have  
$\widetilde{M}_g(n)/\widetilde{M}_g(1)=\Z/n\Z$. This implies that 
$\widetilde{M}_g(n)$ is residually finite. In fact, any element 
of $\widetilde{M}_g(n)$ which is not detected by the 
homomorphism onto $\Z/n\Z$ belongs to $\widetilde{M}_g(1)$. 
Inducting finite groups representations from $\widetilde{M}_g(1)$ 
to $\widetilde{M}_g(n)$ we obtain  finite group representations of the 
later detecting every non-trivial element of $\widetilde{M}_g(1)$.
This proves the claim.

\begin{remark}
Freedman, Walker and Zhang already observed in \cite{FWW} that a simple  
consequence of the  asymptotic faithfulness is  that $M_g$ is 
residually finite. 
\end{remark}

\begin{remark}
This proof works more generally for the punctured mapping 
class group $M_{g}^1$ and for those extensions whose cohomology classes are  
of the form $n\chi+e$, for some $n\in\Z$.  Recall that  
$H^2(M_{g}^1)=\Z\chi\oplus\Z e$, where $\chi$   
is the signature class and 
$e$ is the class associated to the puncture, for $g\geq 4$ (see \cite{KS}). 
\end{remark}

\begin{remark}
Notice that there exist  quantum type 
representations of $Sp(2g,\Z)$, for instance 
those associated to the monodromy of  level $k$ theta functions 
in the $U(1)$ gauge theory (see e.g. \cite{Fu,Fu2}).  
Again these are only projective  unitary representations  
which lift to  unitary representations of some central extension  
$\rho_{Sp, k}:\widetilde{Sp(2g,\Z)}(4)\to U(k^g)$. Here 
$\widetilde{Sp(2g,\Z)}(4)$ is the central extension of $Sp(2g,\Z)$ by $\Z$ 
whose class is 4 times the Maslov class. 
However, these representations factor through the 
integer metaplectic group. Further the  generator of the 
kernel of $\widetilde{Sp(2g,\Z)}(4)\to Sp(2g,\Z)$
acts as the multiplication by a root of unity of order $8$, 
for any level $k$.  Thus  the intersection of 
$\cap_{k\geq 2}\rho_{Sp,k}$ is $2\Z$, and the result of Deligne 
cited above shows that this is sharp. 
\end{remark}

\section{Proof of Proposition \ref{fa}}

We consider  first the following notion introduced by Farb in \cite{Fa2}:

\begin{definition}
Let $n\geq 1$. A group $\Gamma$ has property $FA_{n}$ if any isometric action 
on any $n$-dimensional CAT(0) cell complex $X$ has a fixed point. 
\end{definition}

Observe that property $FA_1$ corresponds to Serre's property $FA$, 
which asks that any action without inversions of $\Gamma$ on a 
real tree should fix a vertex. 
Notice that Kazhdan groups have property $FA$.  
Moreover if a group has property $FA_n$ then it has property $FA_k$ for 
all $k<n$. It is known (see \cite{Fa2}) that a group $\Gamma$ with property 
$FA_{n-1}$ has $n$-integral representation type, namely the eigenvalues 
of matrices in $\rho(\Gamma)$, for 
a homomorphism $\rho:\Gamma\to GL(n,K)$ with $K$ a field, are 
algebraic integers if ${\rm char}(K)=0$.  Moreover, there are only 
finitely many conjugacy classes of irreducible representations 
 of $\Gamma$ into $GL(n,K)$, for an algebraically closed field $K$.

Culler and Vogtmann  proved that $M_g$ has property $FA_1$ 
in \cite{CV}. In \cite{Fa} one asks to estimate the maximal $n=n(g)$ for which 
$M_g$ has property $FA_n$.

There is a version of $FA_n$, namely the strong $FA_n$ (which implies $FA_n$), 
in which one considers complete CAT(0) spaces and  semi-simple actions.
It is proved by Bridson in (\cite{B2}, see also \cite{B1}) 
that $M_g$ has strong $FA_g$. Moreover it is known that 
$M_g$ acts (faithfully if $g>2$) by semi-simple isometries 
on the completion of the Teichm\"uller space 
with the Weil-Petersson metric, which has dimension $6g-6$. Thus $g\leq 
n(g)\leq 6g-7$.

The key point is to relate the property $FA_{n}$ to the finiteness of 
unitary representations. Specifically, we have the following: 

\begin{proposition}\label{faa}
If $\Gamma$ is a finitely generated group with  
property $FA_{n^2-1}$ then the 
representations $\Gamma\to SL(n,\C)$ have  finite image. 
\end{proposition}
\begin{proof}
Let $\overline{\Gamma}$ be the image of $\Gamma$ under 
some homomorphism into $SL(n,\C)$.  A finitely generated subgroup 
$\overline{\Gamma}$ of $SL(n,\C)$ lies into some $SL(n,A)$, 
where $A$ is a finitely generated $\Q$-algebra  
contained in $\C$. Let $\varphi:A\to \overline{\Q}$ be a 
specialization of $A$, which  induces 
a morphism $\varphi:SL(n,A)\to SL(n,\overline{\Q})$. 
The image $\varphi(\overline{\Gamma})$ belongs then to 
some $SL(n,K)$, where $K$ is a finite extension of $\Q$. 

\begin{lemma}
If all specializations $\varphi(\overline{\Gamma})$ are finite then 
$\overline{\Gamma}$ is finite. 
\end{lemma}
\begin{proof}
Jordan's theorem says that there is some  
$f(n)$ such that any finite subgroup of $GL(n,K)$ has a normal 
abelian subgroup of index at most $f(n)$. 
The intersection of all subgroups of $\overline{\Gamma}$ 
of index at most $f(n)$ is then a finite index subgroup $U\subset 
\overline{\Gamma}$ 
such that $\varphi([U,U])=1$ for every specialization $\varphi$. 
Since specializations $\varphi$ separate the points of $A$ we have 
$[U,U]=1$. Therefore there exists a finite index 
normal abelian subgroup $U$ of $\overline{\Gamma}$. 
If $U$ is finite then $\overline{\Gamma}$ will be finite and we are done.

Let us assume from now on that $U$ is infinite.  
Since $U$ is finitely generated abelian there is an infinite order 
element $Z\in U$. 
The following lemma will show that there exists a specialization $\varphi$ 
such that $\varphi(Z)$ is of infinite order, contradicting our assumptions.  
Thus $U$ should be finite abelian and the result will follow. 

\begin{lemma}\label{crucial}
Let $Z$ be a matrix with entries in a finitely generated $\Q$-algebra $A$ 
contained in $\C$. Suppose that for every number field $K$ and any 
ring homomorphism $\varphi:A\to K$ the image $\varphi(Z)$ 
is a matrix of finite order. Then $Z$ has finite order.  
\end{lemma}
\begin{proof}
By Noether's normalization lemma (see \cite{Re}, p.63) 
there exist algebraically independent 
elements  $\xi_1,\xi_2,\ldots, \xi_p\in A$ such that  
$A$ is an integral extension of the purely transcendental 
extension $B=\Q[\xi_1,\ldots,\xi_p]$. 
Moreover  $\xi_1,\ldots,\xi_p$ form a transcendence basis for the 
field of fractions  of $A$ over $\Q$.

Let $\lambda_1,\ldots,\lambda_n$ be the eigenvalues of the
matrix $Z$. We will prove that $\lambda_j$ are roots of unity. 
First $\lambda_j$ are integral over $A$ because they are 
the roots of the characteristic polynomial of $Z$, which is a monic 
polynomial with  coefficients in $A$. The integrality is transitive and hence $\lambda_i$ are 
integral over $B$. Thus $\lambda_j$ satisfies 
an algebraic equation $P_j(\lambda_j)=0$, where $P_j\in B[X]$ is 
the minimal polynomial of $\lambda_j$ over $B$. 
The polynomial $P_j$ is monic and irreducible 
because $B$ is a unique factorization domain. This implies that, 
if we consider $P_j$ as a polynomial from 
$\Q[\xi_1,\ldots,\xi_p,X]$ then it is still an irreducible polynomial 
in the $p+1$ variables $\xi_1,\ldots,\xi_p,X$.

If $p=0$ then the fractions field of $A$ is a number 
field and thus $\lambda_i$ should be roots of unity.

Let us assume henceforth that $p\geq 1$.
Observe that any specialization $\varphi:B\to \overline{\Q}$ 
can be lifted (not uniquely) to a specialization  $\varphi:D\to \overline{\Q}$
of a finite extension $D$ of $B$. 
First, specializations of $B$ can be extended 
to possibly infinite specializations (see \cite{Weil}, Thm.6, p.31) 
of any extension $D$ of $B$. Moreover, the extended specialization is finite 
on any element of $D$ which is integral 
over $B$ (see \cite{Weil}, Prop.22, p.41).  
In particular, any specialization of $B$  
extends to $A[\lambda_1,\ldots,\lambda_n]$. 
On the other hand, observe that any 
specialization $\varphi$ of $B$ corresponds to prescribing  
the values of $\varphi(\xi_j)\in\overline{\Q}$ arbitrarily.

Hilbert's irreducibility theorem states that there exist infinitely many 
(actually a Zariski dense set of) specializations $\varphi:B\to {\Q}$ 
such that the polynomials $\varphi(P_j)\in \Q[X]$ are still irreducible. 
Since $\varphi(Z)$ is of finite order, each 
$\varphi(\lambda_j)$ is a root of unity so that 
$\varphi(P_j)$ should be a cyclotomic polynomial. 
The degree of $\varphi(P_j)$ is the degree  $d_j$ of $P_j$, 
since these are monic polynomials. 
But there are only finitely many cyclotomic polynomials 
of given degree. Let $S$ be the finite family of coefficients of all  
cyclotomic polynomials of degree smaller or equal 
to $\max(d_1,\ldots,d_n)$.   
It suffices then to choose some specialization $\varphi$ 
of $B$ for which one coefficient of some $\varphi(P_j)$ 
does not belong to $S$. For instance it suffices to choose 
a specialization for which some coefficient of $\varphi(P_j)$ is not 
in $\Z$, because cyclotomic polynomials have coefficients 
in $\Z$. 
This is possible unless all 
coefficients of the polynomials $P_j$ are independent 
on the $\xi_1,\ldots,\xi_p$. This might happen only if 
$P_j\in\Q[X]$, namely if  all its coefficients, which are elements of 
$\Q[\xi_1,\ldots,\xi_p]$, are actually constant. But in this case 
all $\lambda_j$ are algebraic integers. 
This contradicts the fact that the transcendence degree 
of the fractions field of $A$ was supposed to be $p\geq 1$. 
Therefore all eigenvalues $\lambda_j$ of $Z$ are 
roots of unity.

An alternative argument is as follows. 
The set of $\C$-valued 
specializations $\varphi:A[\lambda_1,\ldots,\lambda_n]\to \C$ 
is an irreducible affine algebraic variety of dimension $p$ and  
$\lambda_j$ is a rational function on it. 
If $\lambda_j$ is a root of unity for any $\overline{\Q}$-valued 
specialization then $|\lambda_j|$ is identically $1$. 
But a bounded regular function on a irreducible complex 
algebraic variety  should be constant. 
This implies that all $\lambda_j$ are algebraic integers and we conclude 
as above.

Eventually, it suffices to show that $Z$ is diagonalizable.
Consider the Jordan-Chevalley decomposition $Z=D+N$, where 
$D$ is semi-simple, $N$ is nilpotent and 
$DN=ND$.  The entries of the matrices $D$ and $N$ belong to  
the field of fractions of $A$ (see \cite{Ch}, Thm.7, p.71-72).  
Let $a\in A$ be the least common multiple of denominators 
arising in the entries of $D$ and $N$.  
Every specialization $\varphi$ of $A$ with the property that 
$\varphi(a)\neq 0$ extends  uniquely to a specialization, 
still denoted $\varphi$, of the localization of $A$ at $a$. In particular,  
it makes sense to consider $\varphi(D)$ and $\varphi(N)$. 
Therefore $\varphi(Z)=\varphi(D)+\varphi(N)$ 
is the Jordan-Chevalley decomposition of $\varphi(Z)$. 
But the minimal polynomial of $\varphi(Z)$ divides 
$X^s-1$, where $s$ is the order of $\varphi(Z)$. 
This implies that the minimal polynomial has distinct roots and 
so $\varphi(Z)$ is semi-simple. 
The uniqueness of the Jordan-Chevalley decomposition yields then  
$\varphi(N)=0$.  Since this holds for any specialization 
$\varphi$  such that $\varphi(a)\neq 0$ and such 
specializations separate the points of $A$ we 
derive that $N=0$. Thus $Z$ is diagonalizable 
and hence of finite order, as claimed.  
\end{proof}
\begin{remark}
We could also use (\cite{B2}, Prop. 6.1) which says that 
the image in $GL(g,\C)$ of an element of a finitely generated group 
with strong property $FA_g$ has algebraic eigenvalues. However,    
lemma \ref{crucial} can be applied to more general situations, since 
there is no assumption on $Z$.    
\end{remark}
\end{proof}

It suffices now to show that, for any specialization $\varphi$, the 
image $G=\varphi(\overline{\Gamma})$ is finite. 
Observe that, if $\Gamma$ has property $FA_{n-1}$, 
then $G=\varphi(\overline{\Gamma})$ has also property $FA_{n-1}$. 
We will show that: 

\begin{lemma}
Let $K$ be a number field. Then a  finitely generated 
subgroup $G\subset SL(n,K)$ with property $FA_{n^2-1}$ should be finite. 
\end{lemma}
\begin{proof}
We prove that, for any embedding of $K$ into a local field 
$K_v$ the image of $G$ in $SL(n,K_v)$ is precompact.

If  $G\subset SL(n,K)$  is a finitely generated subgroup with 
property $FA_{n-1}$ then its image in $SL(n,K_v)$ is precompact, 
for each non-Archimedian valuation $v$ of $K$. 
In fact, $G$ acts on the Bruhat-Tits 
building associated to $SL(n,K_v)$, which is a $(n-1)$-dimensional 
CAT(0) cell complex. The $G$-action has a fixed point because 
$G$ has property $FA_{n-1}$ and hence $G$ is contained in the stabilizer 
of a vertex, which is a compact subgroup.

In what concerns the Archimedian valuations 
it suffices to consider the complex ones. 
But $SL(n,\C)$ acts on the symmetric space $SL(n,\C)/SU(n)$ 
of non-compact type and real dimension $n^2-1$. Since this space 
is CAT(0) and $G$ has property $FA_{n^2-1}$ 
it follows that the image of $G$ into $SL(n,\C)$ 
is contained in  the stabilizer $U(n)$ for any  complex 
valuation inducing  an embedding $K\to \C$.

Eventually recall that $SL(n,K)$ embeds as a discrete subgroup of 
the special linear group $SL(n, A_K)$ over the ad\`eles ring $A_K$ of $K$. 
By above $G$ is discrete and precompact into $SL(n, A_K)$ and 
hence finite. 
\end{proof}

This proves Proposition \ref{faa}. 
\end{proof}

\begin{remark}
If $G$ is a subgroup of $U(n)\cap SL(n,\Q)$ with property $FA_{n-1}$ 
then $G$ is finite. This follows from above by using the fact that 
there is an unique complex Archimedian valuation on $\Q$, and one knows 
already that $G$ is contained in the compact group $U(n)$. 
In particular, if $\Gamma$ has property $FA_{n-1}$ 
then the image of  any homomorphism $\Gamma\to U(n)\cap SL(n,\Q)$ 
is finite.  
\end{remark}

{\em End of the proof of Proposition \ref{fa}.} 
The result of proposition \ref{faa} holds also for 
representations into $PSL(n,\C)$ and a fortiori for representations 
into $PU(n)$. Since $M_g$ has property $FA_g$ we derive the  
lower bound inequality.

Consider now the smallest (projective) quantum representation 
$M_g\to PU(d_g)$ with infinite image, for $g\geq 2$. 
This is the $SO(3)$ quantum representation in level $5$ 
(see e.g. \cite{F}), whose dimension $d_g$ is given by the Verlinde 
formula: 
\[ d_g=\left(\frac{5}{4}\right)^{g-1}\sum_{j=1}^4 \left(\sin\frac{2\pi j}{5}\right)^{2-2g}=\left\{\begin{array}{ll}
5^{g/2}F_{g-1}, \mbox{ if } g \mbox{ is even},\\
5^{(g-1)/2}(F_{g}+F_{g-2})
\end{array}\right.\]
where $F_j$ is the Fibonacci sequence $F_0=0,F_1=1, F_{n+1}=F_n+F_{n-1}$. 
For instance $d_2=5$. These mapping class group 
representations come from the so-called Fibonacci TQFT.

Moreover, it is clear that the upper bound holds for any finite 
index subgroup of $M_g$. In fact the image of a finite 
index subgroup of the mapping class group by the quantum representation 
is still infinite. This proves the claim.

\begin{remark}
The property $FA_{n-1}$ is  not 
inherited by the finite index subgroups.  
Actually $M_2$ has a finite index subgroup which surjects 
onto a free non-abelian group and hence it has not 
property $FA_1$. 
The situation is subtler for $g\geq 3$ and it seems unknown whether 
finite index subgroups of $M_g$ have property $FA_1$. 
Bridson proved in \cite{B1} that for any  normal subgroup $H$ 
of index $n$ in  $M_g$, for $g\geq 3$,  and any 
homomorphism $\phi:H\to G$ to a group $G$ acting by hyperbolic 
isometries on some complete CAT(0) space -- in particular, to $G=\Z$ --
the $n$-th powers of Dehn twists (which belong to $H$) 
lie in the kernel of $\phi$. Such homomorpisms $\phi$ have therefore 
striking similarities with the quantum representations.   
\end{remark}

Corollary \ref{fin} follows from Proposition \ref{faa} 
above and Bridson's result from \cite{B2} saying   
that $M_g$ has strong $FA_g$.

\acknowledgements{ 
We are indebted to 
Jean-Beno\^it Bost, Benson Farb, Eric Gaudron and Bertrand Remy for helpful 
discussions and to the referees for pointing out an incomplete argument in 
the first version.   
The author was partially supported by 
the ANR Repsurf: ANR-06-BLAN-0311.
}

{
\small      
      
\bibliographystyle{plain}

\begin{thebibliography}{30}      



\bibitem{A}
J.E.Andersen, {\em  Asymptotic faithfulness of the 
quantum ${\rm SU}(n)$ representations of the mapping class groups},   
Ann. of Math. (2)  163(2006),  347--368. 









\bibitem{B1}
M.R.Bridson, {\em 
Semisimple actions of mapping class groups on CAT(0) spaces}, 
arXiv:0908.0685, "The Geometry of Riemann Surfaces", 
London Mathematical Society Lecture Notes 368, 
Dedicated to Bill Harvey on his 65th birthday, to appear. 

\bibitem{B2}
M.R.Bridson, {\em 
On the dimension of CAT(0) spaces where mapping class groups act}, 
arXiv:0908.0690. 


\bibitem{Ch}
C.Chevalley, Th\'eorie des groupes de Lie II. Groupes alg\'briques, 
{\em Actualit\'es Sci. Ind.} 1152, {\em Hermann $\&$ Cie.}, Paris, 1951. 


\bibitem{CV}
M.Culler and K.Vogtmann, {\em A group-theoretic 
criterion for property ${\rm FA}$},   
Proc. Amer. Math. Soc.  124(1996), 677--683.


\bibitem{De}
P.Deligne, {\em  Extensions centrales non r\'esiduellement finies de groupes 
arithm\'etiques}, 
C. R. Acad. Sci. Paris S\'er. A-B 287(1978), no. 4, A203--A208. 
   



\bibitem{Fa}
B.Farb,  {\em Some problems on mapping class groups and moduli space}, 
in   Problems on mapping class groups and related topics,  11--55, Proc. Sympos. Pure Math., 74, Amer. Math. Soc., Providence, RI, 2006. 



\bibitem{FaMa}
B.Farb and D.Margalit, {\em A primer of mapping class groups}, preprint. 






\bibitem{Fa2}
B.Farb, {\em Group actions and Helly's theorem}.
Advances Math. 222(2009), 1574-1588. 

\bibitem{FLM}
B.Farb, A.Lubotzky and Y.Minsky, 
{\em Rank one phenomena in mapping class groups}, 
Duke Math. J. 106(2001), 581-597. 

\bibitem{FWW}
M.H.Freedman, K.Walker and Zhenghan Wang, 
{\em Quantum $\rm SU(2)$ faithfully detects mapping class groups modulo center},  Geom. Topol.  6(2002), 523--539. 


\bibitem{Fu}
L.Funar, {\em 
Repr\'esentations du groupe symplectique et vari\'et\'es de dimension $3$},   
C. R. Acad. Sci. Paris S\'er. I Math.  316(1993),  no. 10, 1067--1072. 

\bibitem{F}
L. Funar, {\em On the TQFT representations of the mapping class groups}, 
Pacific J. Math. 188(1999), 251--274. 



\bibitem{Fu2}
L.Funar, {\em Some abelian invariants of 3-manifolds},  
Rev. Roumaine Math. Pures Appl.  45(2000), 825--861.


\bibitem{Ge}
S.Gervais, {\em 
Presentation and central extensions of mapping class groups},  
Trans. Amer. Math. Soc.  348(1996), 3097--3132. 


\bibitem{KS}
M. Korkmaz and A.I. Stipsicz, {\em
The second homology groups of mapping class groups of oriented surfaces},
Math. Proc. Cambridge Philos. Soc.  134(2003),  479--489.



\bibitem{Mar}
G.Margulis, {Discrete subgroups of  semisimple Lie groups}, 
{\em Ergebnisse Math. Grenz.} 17, Springer Verlag, Berlin, 1991. 



\bibitem{MR}
G.Masbaum and J.D.Roberts, {\em 
On central extensions of mapping class groups},  
Math. Ann.  302(1995), 131--150.

\bibitem{Re}
M.Reid, Undergraduate commutative algebra, {\em 
London Mathematical Society Student Texts} 29, 
Cambridge University Press, Cambridge, 1995. 



\bibitem{Weil}
A.Weil, Foundations of Algebraic Geometry, 
{\em American Mathematical Society Colloquium Publications} 29, 
American Mathematical Society, New York, 1946.


\end{thebibliography}

}

\end{document}